\documentclass[pdflatex,sn-mathphys-num]{sn-jnl}

\usepackage{amssymb}
\usepackage{multirow}
\usepackage{ifthen}
\usepackage{float}
\usepackage[section]{placeins}
\usepackage{graphicx}
\usepackage{longtable}
\usepackage{mathrsfs}
\usepackage{graphicx}%
\usepackage{multirow}%
\usepackage{amsmath,amssymb,amsfonts}%
\usepackage{amsthm}%
\usepackage{mathrsfs}%
\usepackage[title]{appendix}%
\usepackage{xcolor}%
\usepackage{textcomp}%
\usepackage{manyfoot}%
\usepackage{booktabs}%
\usepackage{algorithm}%
\usepackage{algorithmicx}%
\usepackage{algpseudocode}%
\usepackage{listings}%
\usepackage{booktabs}

\floatstyle{ruled}
\newfloat{algorithm}{htb}{alg}
\floatname{algorithm}{Algorithm}

\theoremstyle{thmstyleone}%
\newtheorem{theorem}{Theorem}
\newtheorem{proposition}[theorem]{Proposition}%

\theoremstyle{thmstyletwo}%

\theoremstyle{thmstylethree}%
\newtheorem{definition}{Definition}%
 
\usepackage{xcolor}

\newcommand{\vecc}{}
\newcommand{\te}{\mathcal}

\newcommand{\bcirc}{\operatorname{{\tt bcirc}}}

\newcommand{\fold}{\operatorname{{\tt fold}}}

\newcommand{\unfold}{\operatorname{{\tt unfold}}}

\newcommand{\R}{\mathbb R}
\newcommand{\C}{\mathbb C}
\newcommand{\A}{\mathcal{A}}
\begin{document}

\title[ The tubal Arnoldi method for tensor function approximation]{ The tubal Arnoldi method for tensor function approximation}


\author*[1]{\fnm{Fatima} \sur{Bouyghf}}\email{fatima.bouyghf@enac.fr}
\equalcont{These authors contributed equally to this work.}

\author[2]{\fnm{Mohamed} \sur{El Ghomari}}\email{m.elghomari10@gmail.com}
\equalcont{These authors contributed equally to this work.}

\author[3]{\fnm{Alaa} \sur{ El Ichi}}\email{elichi.alaa@gmail.com}
\equalcont{These authors contributed equally to this work.}

\affil[1]{\orgdiv{OPTIM}, \orgname{\'{E}cole nationale de l'aviation civile, Universit\'{e} Toulouse }, \orgaddress{\street{ 7 avenue Edouard Belin }, \city{Toulouse}, \postcode{31400}, , \country{France}}}

\affil[2]{\orgdiv{Department of Mathematics}, \orgname{\'{E}cole Normale Sup\'{e}rieure, 
		Mohammed V University}, \orgaddress{\street{Takaddoum, B.P. 5118}, \city{Rabat}, \country{Morocco}}}

\affil[3]{\orgdiv{LMPA}, \orgname{Universit\'{e} du Littoral C\^{o}te d'Opale}, \orgaddress{\city{Calais}, \country{France}}}


\abstract{This paper describes Krylov subspace methods based on the tensor t-product for approximating quantities associated with third-order tensor functions. Specifically, it introduces the tensor-tubal Arnoldi method, which projects high-dimensional tensor problems onto lower-dimensional tensor Krylov subspaces. The derived method exploits novel algebraic properties of the t-product to address large-scale tensor computations efficiently. Applications include the evaluation of parameter-dependent tensor functions and the solution of multidimensional ordinary differential equations. Numerical experiments illustrate the effectiveness and accuracy of the proposed methods compared with other approaches.}

\keywords{Tensor Krylov subspaces, t-product, Arnoldi algorithm, Tensor full orthogonalization method, multidimensional ODEs, tensor functions, Stieltjes functions.}


\pacs{65B05, 15A69, 15A72, 65F22}

\maketitle

\section{Introduction}\label{sec1}
Computing the  approximation of    $f(A)b$ where $A \in \mathbb{R}^{n\times n}, b \in \mathbb{R}^{n}$
and $f$ is a function defined on the convex hull of the spectrum of $A$ is an important task in various applications. These include network analysis ($f(x):=\exp(x)$) \cite{EH}, Markov model analysis ($f(x):=\log(x)$) \cite{Singer}, lattice quantum chromodynamics ($f(x)=sign(x)$) \cite{van},  exponential integrators in differential equations ($f_t(x)=\exp{(-tx)}$)  \cite{Hochbruck,Hochbruck1}, and Quantum
Chromodynamics theory $(f(x)=\sqrt{x}$) \cite{BM}. For small matrices $A$, the function $f(A)$ can be defined in terms of the spectral decomposition or the Jordan canonical form \cite{Higham,Higham1}.  However, in many practical applications, $A$ is so large that such approaches become computationally infeasible. In this case, several Krylov subspace methods have been proposed (see, e.g., \cite{BMS1,Hochbruck,Jbilou,Knizhnerman,saad1}). The first method introduced in this context is the standard Arnoldi process, which uses subspaces with nonnegative powers of $A$. Consider the  Krylov subspace 
\begin{equation}\label{Km}
	K_m(A,b)=span\left\lbrace b, Ab,\ldots,A^{m-1}b \right\rbrace,
\end{equation} 
and applying $ m $ steps of the Arnoldi method to $ A $ with initial vector $ b $ yields the following decomposition:

$$
AV_m = V_m H_m + h_{m+1,m} v_{m+1} e_m^T
$$

where $ V_m = [v_1, \ldots, v_m] \in \mathbb{R}^{n \times m} $ is an orthonormal basis for the Krylov subspace \eqref{Km}, and $ H_m \in \mathbb{R}^{m \times m} $ is a Hessenberg matrix; see Saad \cite{saad1} for more details.
The  evaluation of $f(A)b$ via   the Arnoldi process  is  given by 
$$ f(A)b=\|b\|V_mf(H_m)e_1.$$

\noindent 	A tensor is a multidimensional array of data that generalizes matrices. Tensors have emerged as fundamental tools in various computational and applied mathematics domains, including network analysis \cite{CRZT} and multidimensional differential equations \cite{Khoromskij}. The tensor t-product, originally proposed by Kilmer and collaborators \cite{KM,MSL}, extends the classical notions of matrix?matrix and matrix?vector multiplication to third-order tensors.
The t-product has found applications in image processing \cite{EEJS,klimer3,MSL,RU1}, signal processing \cite{LLCZ,SHSW}, tensor function approximation \cite{Lund,BEJRR} and data completion and denoising \cite{BEJR}.
Lund \cite{Lund} defined tensor functions within the framework of the t-product for third-order $f$-square tensors, whose frontal slices are square matrices. The frontal slices of these tensors are square matrices and can be defined in terms of their spectral decomposition \cite{Lund} or the Jordan canonical form \cite{MQW2}.

\noindent In this paper, we introduce the tensor t-tubal Arnoldi method to evaluate the following tensor problem:
\begin{equation}\label{eq1}
	\mathcal{I}(f) := f(\mathcal{A}) \star \mathcal{B}
\end{equation}
where $ \mathcal{A} \in \mathbb{R}^{n \times n \times n_3} $ and $ \mathcal{B} \in \mathbb{R}^{n \times 1 \times n_3} $, and the function $ f $ is such that the matrix $ f(\bcirc(\mathcal{A})) $ is well-defined.
This method is based on projecting the initial problem \eqref{eq1} onto an increasing sequence of Krylov subspaces via the t-product, transforming the large-scale problem into a reduced problem that can be easily solved by direct methods, with the solution serving as an approximation to the solution of the initial problem.
The proposed method is also applied to approximate parameter-dependent tensor functions of the form
\begin{equation}\label{expA}
	\mathcal{U}(t):=\exp(-t\mathcal{A})\star \mathcal{B},\quad t>0.
\end{equation}

The paper is organized as follows. In section 2, we remind the definition of the matrix Stieltjes functions. Section 3 provides a brief overview of the tensor t-product and several fundamental algebraic properties of third-order tensors. It also recalls the definition of tensor functions as introduced by Lund \cite{Lund}. Section 4 introduces the tensor t-tubal Arnoldi process applied to a class of Stieltjes functions for tensors. Section 5 describes the application
of this process to the approximation of the matrix exponential of the form \eqref{expA}.  Section 5 presents numerical experiments that demonstrate the quality of the computed approximations.
\section{Matrix Stieltjes functions}

Let \( A \in \mathbb{C}^{n \times n} \) be a matrix with spectrum 
\[
\text{spec}(A) := \{ \lambda_j \}_{j=1}^N, 
\]
where \( N \leq n \) and the \( \lambda_j \) are distinct. An \( m \times m \) Jordan block \( J_m(\lambda) \) of an eigenvalue \( \lambda \) has the form
\[
J_m(\lambda) = 
\begin{bmatrix}
	\lambda & 1 & 0 & \cdots & 0 \\
	0 & \lambda & 1 & \cdots & 0 \\
	\vdots & \ddots & \ddots & \ddots & \vdots \\
	0 & \cdots & 0 & \lambda & 1 \\
	0 & \cdots & \cdots & 0 & \lambda
\end{bmatrix} \in \mathbb{C}^{m \times m}.
\] 	
Suppose that \( A \) has Jordan canonical form
\begin{equation}
	A = X J X^{-1} = X^{-1} \operatorname{diag}(J_{m_1}(\lambda_{j_1}), \dots, J_{m_p}(\lambda_{j_p})) X, 
\end{equation}
with \( p \) blocks of sizes \( m_i \), such that \( \sum_{i=1}^p m_i = n \), and where the values \( \{ \lambda_{j_k} \}_{k=1}^p \subseteq \text{spec}(A) \). Note that eigenvalues may be repeated in the sequence \( \{ \lambda_{j_k} \}_{k=1}^p \). Let \( n_j \) denote the index of \( \lambda_j \), that is, the size of the largest Jordan block associated with \( \lambda_j \).  
A function is said to be defined on the spectrum of \( A \) if all of the following quantities exist:

\[
f^{(k)}(\lambda_j),  \quad j = 1, \dots, N,\quad k = 0, \dots, n_j - 1.
\]
\begin{definition}\label{jordan}
	Let \( A \in \mathbb{C}^{n \times n} \) be a matrix with the Jordan canonical form given in (1), and assume that \( f \) is defined on the spectrum of \( A \). We then define
	
	\[
	f(A) := X f(J) X^{-1},
	\]
	where 
	\[
	f(J) := \operatorname{diag}(f(J_{m_1}(\lambda_{j_1})), \dots, f(J_{m_p}(\lambda_{j_p}))).
	\]
	\[
	f(J_{m}(\lambda_{j_k})) := 
	\begin{bmatrix}
		f(\lambda_{j_k}) & f'(\lambda_{j_k}) & \frac{f''(\lambda_{j_k})}{2!} & \cdots & \frac{f^{(m-1)}(\lambda_{j_k})}{(m-1)!} \\
		0 & f(\lambda_{j_k}) & f'(\lambda_{j_k}) & \cdots & \frac{f^{(m-2)}(\lambda_{j_k})}{(m-2)!} \\
		0 & 0 & f(\lambda_{j_k}) & \cdots & \frac{f^{(m-3)}(\lambda_{j_k})}{(m-3)!} \\
		\vdots & \vdots & \ddots & \ddots & \vdots \\
		0 & 0 & \cdots & 0 & f(\lambda_{j_k})
	\end{bmatrix}
	\in \mathbb{C}^{m \times m}.
	\] 	
	Note that if \( A \) is diagonalizable with \( \mathrm{spec}(A) = \{ \lambda_j \}_{j=1}^n \) (not necessarily distinct), then Definition~\ref{jordan} reduces to
	
	\[
	f(A) = X \operatorname{diag}(f(\lambda_1), \dots, f(\lambda_n)) X^{-1}.
	\] 	
	Since matrix powers are well defined, any scalar polynomial can be naturally evaluated on a matrix. Specifically, for 
	\[
	p(z) = \sum_{k=0}^{m} c_k z^k, \quad c_k \in \mathbb{C},
	\]
	we set
	\[
	p(A) := \sum_{k=0}^{m} c_k A^k.
	\]
	This construction provides a basis for defining nonpolynomial functions of matrices, by extending the approach via derivatives as in Definition~\ref{jordan}.
\end{definition} 
\begin{definition}
	Assume that \( f \) is defined on \( \mathrm{spec}(A) \), and let \( p \) be the unique Hermite interpolating polynomial of degree at most \( \sum_{j=1}^N n_j \) satisfying
	
	\[
	p^{(k)}(\lambda_j) = f^{(k)}(\lambda_j), \quad \text{for all } k = 0, \dots, n_j - 1, \quad j = 1, \dots, N.
	\] 	
	We then define \( f(A) := p(A) \).
\end{definition} 	
Let \( \mathbb{D} \subset \mathbb{C} \) be an open domain, and let \( f : \mathbb{D} \to \mathbb{C} \) be an analytic function. We say that \( f \) admits a \emph{Cauchy--Stieltjes integral representation} if there exist
a contour \( \Gamma \subset \mathbb{C} \setminus \bar{\mathbb{D} }\), and
a function \( g : \Gamma \to \mathbb{C} \), integrable along \( \Gamma \),
such that for all \( z \in \mathbb{D} \), we have:
\begin{equation}
	\label{reimaninteg}
	f(z) = \int_{\Gamma} \frac{g(t)}{t - z} \, dt.
\end{equation}
This representation is valid for all \( z \in \mathbb{D} \), i.e., outside the support of \( g \). \\

Now let \( A \in \mathbb{C}^{n \times n} \) be a matrix such that its spectrum \( \sigma(A) \subset \mathbb{D} \). 	The Cauchy-Stieltjes integral formula of the function \( f \) can be extended to matrices through the operator-valued integral:
\begin{equation} 	
	f(A) := \int_{\Gamma} g(t) (tI - A)^{-1} \, dt,
\end{equation}
and the integral is defined in the sense of matrix-valued functions.

The matrix \( (tI - A) \) is invertible if and only if \( t \notin \sigma(A) \), that is,
$
(tI - A)^{-1} $ exists  if and only if $t \in \rho(A),
$
where \( \rho(A) := \mathbb{C} \setminus \sigma(A) \). For the integral defining \( f(A) \) to be well-defined, the integration path \( \Gamma \) must lie entirely within \( \rho(A) \), i.e., must avoid the spectrum of \( A \).

In such cases, classical definitions such as \( f(A) = V f(D) V^{-1} \), valid for diagonalizable matrices \( A = VDV^{-1} \), cannot be applied. Moreover, methods based on the Jordan canonical form may be numerically unstable or computationally expensive for large-scale matrices.	
The Cauchy--Stieltjes integral formula
provides a general and flexible approach to define \( f(A) \), as long as the function \( f \) admits an integral representation and the integration contour \( \Gamma \) lies in the resolvent set of \( A \), i.e., \( \Gamma \cap \sigma(A) = \emptyset \).

This representation is widely used in numerical linear algebra, especially for computing matrix functions such as \( \log(A) \), \( A^\alpha \), or \( \exp(A) \), and forms the basis of various numerical methods, including rational Krylov subspace approximations and quadrature-based techniques.
See, e.g., \cite{Henrici,Lund} for more details about Stieltjes functions.

\section{The tensor t-function }
In  this section, we describes  briefly the notion of   t-function   based on T-product introduced by Lund  \cite{Lund}.
But first, we briefly review the tensor t-product for third-order tensors and some of its fundamental properties (see \cite{klimer3,KM} for further details).  
Let \( \mathcal{A} \in \mathbb{R}^{m \times n \times n_3} \) be a third-order tensor with frontal slices \( \mathcal{A}^{(k)} \in \mathbb{R}^{m \times n} \) for \( k = 1, \dots, n_3 \).
The operations \( \operatorname{{\tt bcirc}} \), \( \operatorname{{\tt unfold}} \), and \( \operatorname{{\tt fold}} \) are defined as follows:

\[
{\tt bcirc}(\mathcal{A}) := \left[\begin{array}{ccccc}
	\mathcal{A}^{(1)} & \mathcal{A}^{(n_3)} & \mathcal{A}^{(n_3-1)} & \cdots & \mathcal{A}^{(2)} \\
	\mathcal{A}^{(2)} & \mathcal{A}^{(1)} & \mathcal{A}^{(n_3)} & \cdots & \mathcal{A}^{(3)} \\
	\vdots & \ddots & \ddots & \ddots & \vdots \\
	\vdots & \ddots & \ddots & \ddots & \vdots \\
	\mathcal{A}^{(n_3)} & \mathcal{A}^{(n_3-1)} & \ldots & \mathcal{A}^{(2)} & \mathcal{A}^{(1)}
\end{array}\right],
\qquad
\operatorname{{\tt unfold}}(\mathcal{A}) = \left[\begin{array}{c}
	\mathcal{A}^{(1)} \\
	\mathcal{A}^{(2)} \\
	\vdots \\
	\vdots \\
	\mathcal{A}^{(n_3)}
\end{array}\right],
\]
and \( \operatorname{{\tt fold}}(\operatorname{{\tt unfold}}(\mathcal{A})) := \mathcal{A} \).

The t-product 
$\mathcal{A}\star\mathcal{B}\in\R^{m\times s\times n_3}$ of $\mathcal{A} \in \mathbb{R}^{m\times n\times n_3}$ and 
$\mathcal{B}\in\mathbb{R}^{n\times s\times n_3}$ is given by 
\[
\mathcal{A}\star\mathcal{B}:=\operatorname{\tt fold}(\operatorname{\tt bcirc}(\mathcal{A})
\operatorname{\tt unfold}(\mathcal{B})).
\]

Introduce the tensor $\bar{\mathcal{A}}\in\mathbb{C}^{m\times n\times n_3}$ whose frontal slices are the diagonal blocks $A_1,\ldots,A_{n_3},$ i.e., 
\[
\bar{\mathcal{A}}=\operatorname{\tt fold}\left[\begin{array}{c}
	A_1\\
	A_2\\
	\vdots \\ 
	A_{n_3}
\end{array}\right].
\] 

\noindent The tensor \( \bar{\mathcal{A}} \) can be obtained by applying the Fast Fourier Transform (FFT) to all the tubes of the tensor \( \mathcal{A} \). This can be done using the Matlab command
\[
\bar{\mathcal{A}} = \texttt{fft}(\mathcal{A}, [ \ ], 3), \quad \text{and} \quad \texttt{ifft}(\bar{\mathcal{A}}, [ \ ], 3) = \mathcal{A},
\]
where \texttt{ifft} denotes the Inverse Fast Fourier Transform.

%
%
%
Note that the operators \texttt{fold}, \texttt{unfold}, and \texttt{bcirc} are linear.\\ 
The conjugate transpose of $\mathcal{A}$  is obtained by transposing each of the frontal slices and then reversing the order of transposed frontal slices 2 through $n_3$. \\
For tensors with \(n \times n\) frontal slices, there is a tensor identity \(\mathcal{I}_{n \times n \times n_3} \in \mathbb{C}^{n \times n \times n_3}\). The tensor \( \mathcal{I}_{n \times n \times n_3} \) has its first frontal slice equal to the \( n \times n \) identity matrix, while all remaining frontal slices are zero matrices. Using \( \mathcal{I}_{n \times n \times n_3} \), one can define the notion of an inverse with respect to the t-product: two tensors \( \mathcal{A}, \mathcal{B} \in \mathbb{C}^{n \times n \times n_3} \) are inverses of each other if 
\[
\mathcal{A} \star \mathcal{B} = \mathcal{I}_{n \times n \times n_3} \quad \text{and} \quad \mathcal{B} \star \mathcal{A} = \mathcal{I}_{n \times n \times n_3}.
\]
Furthermore, the t-product framework allows for a natural definition of tensor polynomials, with powers defined recursively as
\[
\mathcal{A}^j := \underbrace{\mathcal{A} \star \cdots \star \mathcal{A}}_{j \text{ times}}.
\]

The inner product $\mathcal{A},\mathcal{B}\in\R^{m\times n\times n_3}$ is defined by
$$\langle \mathcal{A}, \mathcal{B} \rangle = \displaystyle \sum_{i_1=1}^{m} \sum_{i_2=1}^{n}  \sum_{i_3=1}^{n_3} \mathcal{A}_{i_1 i_2 i_3}\mathcal{B}_{i_1 i_2 i_3}.$$
\\ The Frobenius norm of $\mathcal{A}\in\R^{m\times n\times n_3}$ is defined by
$$ \Vert \mathcal{A} \Vert_F=\displaystyle \sqrt{\langle  \mathcal{A} ,  \mathcal{A}  \rangle}.$$
\\ An $n\times n \times n_{3}$ tensor  $\mathcal{Q}$  is orthogonal if
$$\mathcal{Q}^{T}   \star  \mathcal{Q}=\mathcal{Q} \star \mathcal{Q}^{T}=\mathcal{I}_{ n \times n \times n_{3}}.$$		
\\	A tensor is called f-diagonal if its frontal slices are orthogonal matrices. It is called f-upper triangular if all its frontal slices are upper triangular.

As motivation, Lund \cite{Lund} considers the solution of a multidimensional ordinary differential equation.  
Let \( \mathcal{A} \in \mathbb{C}^{n \times n \times n_3} \) be a third-order tensor with square frontal slices, and let \( B: [0,\infty) \rightarrow \mathbb{C}^{n \times n \times n_3} \) be an unknown tensor-valued function with \( B(0) \) prescribed. We then consider the differential equation

\begin{equation}
	B^\prime(t)=\mathcal{A}\star B(t).
\end{equation}

By unfolding both sides, we obtain

\[\begin{bmatrix}\left( B^\prime(t)\right) ^{(1)}\\ \vdots\\ \left( B^\prime(t)\right) ^{(n)}\end{bmatrix}={\tt bcirc}(\mathcal{A})\begin{bmatrix}B^{(1)}(t)\\ \vdots\\ B^{(n)}(t)\end{bmatrix},\]

whose solution can be represented using the matrix exponential as

\[\begin{bmatrix}B^{(1)}(t)\\ \vdots\\ B^{(n)}(t)\end{bmatrix}=\exp(\bcirc(\mathcal{A})\,t)\begin{bmatrix}B^{(1)}(0)\\ \vdots\\ B^{(n)}(0)\end{bmatrix}.\]

By folding both sides again, one obtains the \emph{tensor t-exponential},
\begin{equation}\label{Expo}
	B(t)=\fold(\exp(\mathcal{A}\,t)\unfold(B(0)))=:\exp(\mathcal{A}\,t)*B(0).
\end{equation}

Inspired by the tensor t-exponential, and assuming that \( f(\mathrm{bcirc}(\mathcal{A})) \) is well-defined, we may introduce a generalized definition for the action of a scalar function \( f \) on a tensor \( \mathcal{A} \in \mathbb{C}^{n \times n \times n_3} \) on another tensor \( \mathcal{B} \in \mathbb{C}^{n \times s \times n_3} \), given by

\begin{equation} \label{tensor function}
	f(\mathcal{A}) \star \mathcal{B} := {\rm fold}\left(f({\rm bcirc}(\mathcal{A})) \, {\rm unfold}(\mathcal{B})\right).
\end{equation}

\medskip
	%
%

\begin{definition}\cite{Lund}\label{Tegndecomp}(Tensor eigndecomposition)
	Let $\mathcal{A}$  be a  tensor in  $ \mathbb{C}^{n\times n\times n_{3}}$ which al its frontal slices  are  diagonalizable. Then, they exist  a tensor  $\mathcal{P}$ and  f-diagonal tensor $\mathcal{D}$     such that  
	\begin{equation}
		\mathcal{A}=\mathcal{P}\star\mathcal{D}\star\mathcal{P}^{-1}
	\end{equation}
	and  we  can  write 
	\begin{equation}\label{vpvp}
		\mathcal{A}\star\overrightarrow{\mathcal{P}} _i= \overrightarrow{\mathcal{P}} _i\star{\bf d}_i\;\; i=1,\ldots,n_3 
	\end{equation}
	where  $\overrightarrow{\mathcal{P}}_i$ denote  de ith lateral  slace  of  the  tensor $\mathcal{P}$   and  ${\bf d}_i$ the  ith  diagonal tube  fiber if  the  tensor $\mathcal{D}$.  We  called the  tube fiber ${\bf d}_i$ the  ith  tubal  eigenvalue  of $\mathcal{A}$  and  $\overrightarrow{\mathcal{P}}_i$ the associeted tubal  eigenvector.   
	
\end{definition}
\begin{theorem}\cite{Lund}
	Let $\mathcal{A}$  be a  tensor in  $ \mathbb{K}^{n\times n\times n_{3}}$  and $ f: \mathbb{C}\mapsto \mathbb{C}  $ be defined on a region in the complex plane containing the spectrum of $\mathcal{A}$. We  have
	\begin{enumerate}
		\item [1] $f (\mathcal{A}^T)=f (\mathcal{A})^T$
		\item[2] If   $\mathcal{A}=\mathcal{P}\star\mathcal{D}\star\mathcal{P}^ {-1} $ then  $f(\mathcal{A})=\mathcal{P}\star f (\mathcal{D})\star\mathcal{P}^ {-1}$ 
		\item [3]    If $\mathcal{A}\star \overrightarrow{\mathcal{P}}  = \overrightarrow{\mathcal{P}} \star{\bf d} $ then   $f(\mathcal{A})\star \overrightarrow{\mathcal{P}}  = \overrightarrow{\mathcal{P}} \star f({\bf d}). $   	 	 	\end{enumerate}
\end{theorem}
To make the tensor t-function practically usable, it is essential to develop efficient numerical methods for approximating expressions of the form \( f(\mathcal{A}) \star\mathcal{B} \) with $\mathcal{A}\in\mathbb{C}^{n\times n\times n_3}$ and $\mathcal{B}\in\mathbb{C}^{n\times m\times n_3}$.  While approaches such as the t-eigendecomposition and t-Krylov methods proposed by Kilmer et al.~\cite{klimer3} are viable, performing a full eigendecomposition may become computationally prohibitive for large-scale tensors. Moreover, constructing specialized t-Krylov schemes might be redundant, due to the known equivalence between the \( f(\mathcal{A}) \star \mathcal{B} \) and \( f(A)B \) problems. In this work, we are interested in case where $m=1$. The following section provide a tensor method for tensor Stieltjes type method.

\section{ The tensor Arnoldi  method for Stieltjes functions of tensors}

In this section, we describe the tensor t-tubal Arnoldi method for generating an orthonormal basis for the tensor tubal Krylov subspace, denoted by \( \mathcal{TK}_{m}(\mathcal{A}, \mathcal{V}) \). This subspace is defined as
\begin{align} \label{ttgk}
	\mathcal{TK}_m(\mathcal{A}, \mathcal{V}) &= \text{Tubal-Span}\left\lbrace {\mathcal{V}}, \mathcal{A} \star {\mathcal{V}}, \ldots, \mathcal{A}^{m-1} \star {\mathcal{V}} \right\rbrace \\
	&= \left\lbrace \mathcal{Z} \in \mathbb{R}^{n \times 1 \times n_3}, \mathcal{Z} = \sum_{i=1}^m \left( \mathcal{A}^{i-1} \star {\mathcal{V}} \right) \star \mathbf{a}_i \right\rbrace \nonumber
\end{align}
where \( \mathbf{a}_i \in \mathbb{R}^{1 \times 1 \times n_3} \). The following algorithm \cite{klimer3} applied to \( \mathcal{V} \in \mathbb{R}^{n \times 1 \times n_3} \) allows us to obtain an orthonormal tensor \( \mathcal{Q} \in \mathbb{R}^{n \times 1 \times n_3} \), and also yields the decomposition
\[
\mathcal{V} = \mathcal{Q} \star \mathbf{a}.
\]

\begin{algorithm}[H]
	\caption{The normalization algorithm \cite{klimer3}}\label{normalization12}
	\textbf{Input:} $ {\mathcal{A}} \in {\mathbb R}^{n\times 1\times n_3}$.
	\begin{enumerate}
		\item Set $\mathcal{\widetilde{A}}={\tt fft}(\mathcal {A},[ ],3)$.
		\item for $j=1:n_3$
		\begin{enumerate}
			\item  ${\rm \bf a}^{(j)}= 				 	 ||{A}^{ (j)}||_2$.
			\item if ${\rm \bf a}^{(j)}<tol$
			\begin{itemize}
				\item[] Stop
				\item[] else
				\item[] ${\mathcal {Q}}^{(j)}=\frac{ { A}^{(j)} }{{\rm \bf a}^{(j)}} $.
				\item[] endif 
			\end{itemize}
			\item $ \mathcal{Q}  ={\tt ifft}(\mathcal {Q},[ ],3) $, $ {\rm \bf a}  ={\tt ifft}({\rm \bf a},[ ],3) $.
		\end{enumerate}
	\end{enumerate}
	\textbf{Outputs:} ${\mathcal{Q}}\in {\mathbb R}^{n\times 1\times n_3 }$ and  ${\rm \bf a}\in {\mathbb R}^{1\times 1 \times n_3 }$.
\end{algorithm}

To generate a new basis element \( \mathcal{V}_{j+1} \) in the tensor t-tubal Arnoldi procedure, one must explicitly orthogonalize it against all existing basis tensors \( \mathcal{V}_1, \mathcal{V}_2, \ldots, \mathcal{V}_j \). This process is carried out through a recurrence relation given by:

\begin{equation}\label{VjA}
	\begin{array}{rcl}
		\mathcal{V} &=& \mathcal{V}_1 \star \mathbf{r}_1\\
		\mathcal{V}_{j+1}\star \mathbf{h}_{j+1,j}&=&\mathcal{A}\star\mathcal{V}_j-\sum\limits_{i=1}^j \mathcal{V}_i\star \mathbf{h}_{i,j}=:\mathcal{W}_{j+1},\quad j=1,2,\ldots,k.
	\end{array}
\end{equation}
The tubal coefficients $\mathbf{h}_{i,j} \in \mathbb{R}^{1 \times 1 \times n_3}$ are computed to ensure that the tensors $\mathcal{V}_1, \mathcal{V}_2, \ldots, \mathcal{V}_{j+1}$ are orthonormal. Specifically, these coefficients are determined as follows:
\[
{\rm \bf h}_{i,j}=  {\mathcal{V}}_i^T\star \mathcal{A}\star\mathcal{V}_j,\quad 1\leq i\leq j,\quad
\mathcal{W}_{j+1}=\mathcal{V}_{j+1}\star{\rm \bf h}_{j+1,j}.
\]

\noindent The tensor t-tubal Arnoldi method is summarized in Algorithm \ref{TTA}.

\begin{algorithm}[H]
	\caption{The tensor t-tubal Arnoldi algorithm} \label{TTA}
	\textbf{Inputs:} $\mathcal{A}\in \mathbb{R}^{n\times n\times n_3 }$, ${\mathcal{V}}\in \mathbb{R}^{n\times 1\times n_3 }$ and the number of steps $m$.
	\begin{enumerate}
		\item Set $[\mathcal{V}_{1}, {\rm \bf r}_{1}]=  \text{Normalization}(\mathcal{V})$.
		\item for $j=1:m$
		\begin{enumerate}
			\item  ${\mathcal{W}}=  \mathcal{A}\star   {\mathcal{V}}_j$.
			\item for $i=1:j$
			\begin{itemize}
				\item[] $ {\rm \bf h}_{i,j}=  {\mathcal{V}}_i^T\star {\mathcal{W}}  $.
				\item[]  $\mathcal{W}=\mathcal{W}-   {\mathcal{V}}_i\star{\rm \bf h}_{i,j}$.
				\item[]endfor 
			\end{itemize}
			\item $[{\mathcal{V}}_{j+1}, {\rm \bf h}_{j+1,j}]=  \text{Normalization}({\mathcal{W}} )$.
			\item endfor
		\end{enumerate} 
	\end{enumerate}
	\textbf{Output:} The tensor t-tubal Arnoldi decomposition \eqref{decompp}.
\end{algorithm}

\noindent It follows from the recursion formulas of Algorithm \ref{TTA} the decomposition
\begin{align}
	\label{decompp}
	\mathcal{A} \star \mathbb{V}_m = \mathbb{V}_m \star \mathcal{H}_m + \mathcal{V}_{m+1} \star (\mathbf{h}_{m+1,m} \star \mathcal{E}_m),
\end{align}
where
\[
\mathbb{V}_m := \left[ \mathcal{V}_1, \ldots, \mathcal{V}_m \right] \in \mathbb{R}^{n \times m \times n_3}, \quad \mathbb{V}_m^T \star \mathbb{V}_m = \mathcal{I}_{m \times m \times n_3}.
\]
The tensor \( \mathcal{H}_m \) is defined as
\[
\mathcal{H}_m = \left[ \begin{array}{*{20}{c}}
	\mathbf{h}_{1,1} & \mathbf{h}_{1,2} & \cdots & \mathbf{h}_{1,m} \\
	\mathbf{h}_{2,1} & \mathbf{h}_{2,2} & \cdots & \mathbf{h}_{2,m} \\
	& \ddots & \ddots & \vdots \\
	& & \mathbf{h}_{m,m-1} & \mathbf{h}_{m,m}
\end{array} \right] \in \mathbb{R}^{m \times m \times n_3}.
\]
The tensor \( \mathcal{E}_m = \left[ \mathcal{O}_{1 \times 1 \times n_3}, \ldots, \mathcal{O}_{1 \times 1 \times n_3}, \mathbf{e} \right] \in \mathbb{R}^{1 \times m \times n_3} \) where \( \mathcal{O}_{1 \times 1 \times n_3} \) is the zero tube fiber and \( \mathbf{e} \in \mathbb{R}^{1 \times 1 \times n_3} \) is the tube fiber such that \( \unfold(\mathbf{e}) = (0, 0, 0, \ldots, 1)^T \). Moreover, the tensor \( \mathcal{H}_m \) satisfies \( \mathcal{H}_m = \mathbb{V}_m^T \star A \star \mathbb{V}_m \). The   Krylov subspace  $\mathcal{TK} _{m}(\mathcal{A},{\mathcal{V}} )$ defined in equation (\ref{ttgk}) can be also expressed  as 
\begin{equation}\label{krylovexpr2}
	\mathcal{TK}_{m}(\mathcal{A},{\mathcal{V}}) =\left\lbrace {\mathcal{X}}\in  \mathbb{R}^{n\times 1 \times n_3} /
	{\mathcal{X}}=\mathbb{V}_m \star{\mathcal{Y}} \text{ where } {\mathcal{Y}} \in \mathbb{R}^{m\times 1 \times n_3}\right\rbrace .
\end{equation}

\medskip
\noindent

We are interested in the approximation of \( \mathcal{I}(f) \) in \eqref{eq1} using the tensor t-tubal Arnoldi method, where \( f \) has a Stieltjes integral representation.   To do this, we consider the following result.    
\begin{proposition}
	Let \( \mathcal{A} \in \mathbb{C}^{n \times n \times n_3} \), and let \( \mathcal{I}_{n \times n \times n_3} \) be the identity tensor under the t-product. Then the tensor \( \mathcal{A} - t \mathcal{I}_{n \times n \times n_3} \) is invertible under the t-product if and only if the scalar \( t \in \mathbb{C} \) does not belong to the spectrum of any of the frontal slices \( \bar{\mathcal{A}}^{(k)} \) of the Fourier-transformed tensor \( \bar{\mathcal{A}}\). In other words,
	\[
	\mathcal{A} - t \mathcal{I}_{n \times n \times n_3} \text{ is invertible } \iff \forall k = 1, \dots, n_3,\quad t \notin \sigma(\bar{\mathcal{A}}^{(k)}).
	\]
\end{proposition}

\begin{proof}
	Applying the discrete Fourier transform along the third dimension gives
	\[
	\bar{\mathcal{A}} = \texttt{fft}(\mathcal{A}, [], 3), \quad
	\bar{\mathcal{I}}_{n \times n \times n_3} = \texttt{fft}(\mathcal{I}_{n \times n \times n_3}, [], 3),
	\]
	and it is known that \( \bar{\mathcal{I}}^{(k)} = I_{n\times n} \) for all \( k = 1, \dots, n_3 \). Therefore,
	\[
	\texttt{fft}(\mathcal{A} - t \mathcal{I}_{n \times n \times n_3}) = \bar{\mathcal{A}} - t \bar{\mathcal{I}}_{n \times n \times n_3}= \left\{ \bar{\mathcal{A}}^{(k)} - t I_{n \times n } \right\}_{k=1}^{n_3}.
	\]
	By the definition of the t-product, the tensor \( \mathcal{A} - t \mathcal{I}_{n \times n \times n_3} \) is invertible if and only if each matrix \( \bar{\mathcal{A}}^{(k)} - t I _{n \times n }\) is invertible. This holds if and only if \( t \notin \sigma(\bar{\mathcal{A}}^{(k)}) \) for all \( k = 1, \dots, n_3 \), where \( \sigma(\cdot) \) denotes the set of eigenvalues (the spectrum).
\end{proof}

Given a Stieltjes function \( f \), defined on the complex plane and excluding non-positive real values, which contains the t-eigenvalues of \( \mathcal{A} \). Then under the invertibility condition of tensor $ \mathcal{A} - t \mathcal{I}_{n \times n \times n_3}$ and according to equation \eqref{reimaninteg}, \( \mathcal{I}(f) \) for the class of Stieltjes functions can be extended as follows:

\begin{align*}
	\mathcal{I}(f)=f(\mathcal{A})\star {\mathcal{B}}=
	\int_{0}^{\infty} (\mathcal{A}+t\;\mathcal{I}_{n\times n\times n_3} )^{-1}\star{\mathcal{B}}\;d\mu(t).
\end{align*}
Observe that \( \mathcal{I}(f) \) is closely related to the following shifted tensor linear system

\begin{align}
	\label{shs}(\mathcal{A}+t\mathcal{I}_{n\times n\times n_3})\star\mathcal{X}(t)=\mathcal{B},\quad t\in(0,\infty),\quad \mathcal{X}(t)\in\mathbb{R}^{n\times 1\times n_3}.   
\end{align}

It is natural to approximate the quantity \( \mathcal{I}(f) \) by

\begin{align}\label{FFm}\mathcal{F}_m:=\int_{0}^{\infty}  {\mathcal{X}}_m(t)  \;d\mu(t),
\end{align}
where \( {\mathcal{X}}_m(t) \in \mathbb{R}^{n \times 1 \times n_3} \) can be interpreted as the approximate solution of \eqref{shs}. Let \( {\mathcal{R}}_0(t) \) be the corresponding residual associated with the initial guess \( \mathcal{X}_0 \), i.e., 
\[
{\mathcal{R}}_0(t) := {\mathcal{B}} - (\mathcal{A} + t \mathcal{I}_{n \times n \times n_3}) \star {\mathcal{X}}_0.
\]
Carry out \( m \) steps of the tensor tubal Arnoldi method on the pair \( (\mathcal{A}, \mathcal{R}_0(t)) \). The expression for \( \mathcal{X}_m(t) \) can be written as follows:
\begin{equation}
	\label{fom1}
	{\mathcal{X}}_{m}(t):={\mathcal{X}}_{0}+\mathbb{V}_m \star \mathcal{Z}_m(t), 
\end{equation}  
\(\mathcal{Z}_m(t) \in \mathbb{R}^{m \times 1 \times n_3}\) is determined such that the new residual \(\mathcal{R}_m(t) = \mathcal{B} - (\mathcal{A} + t  \mathcal{I}_{n\times n\times n_3}) \star \mathcal{X}_m(t)\) associated with \(\mathcal{X}_m(t)\) is t-orthogonal to \(\mathcal{TK}_m(\mathcal{A}, \mathcal{R}_0)\), i.e., \(\mathcal{R}_m(t) \perp \mathcal{TK}_m(\mathcal{A}, \mathcal{R}_0)\). This yields
\begin{equation}
	\label{fom2} 
	\mathbb{V}_m^T \star \mathcal{R}_m(t) = 0.
\end{equation}
According to equation \eqref{decompp}, we obtain the shifted decomposition
\begin{align*}
	(\mathcal{A}+t\,\mathcal{I}_{n\times n\times n_3}) \star \mathbb{V}_m = \mathbb{V}_m \star (\mathcal{H}_m + t\, \mathcal{I}_{m\times m\times n_3}) + \mathcal{V}_{m+1} \star \left( {\tt h}_{m+1,m} \star \mathcal{E}_m \right).
\end{align*}
Starting from equation \eqref{fom2} and using \eqref{fom1} along with the shifted decomposition, we get
\[
(\mathcal{H}_m + t\, \mathcal{I}_{m\times m\times n_3}) \star \mathcal{Z}_m(t) = \mathbb{V}_m^T \star \mathcal{R}_0(t).
\]
Finally, the reduced tensor linear system can be written as
\begin{align}\label{reSy}
	(\mathcal{H}_m + t\, \mathcal{I}_{m\times m\times n_3}) \star \mathcal{Z}_m(t) = \mathcal{E}_1^{(m)} \star {\rm \bf r}_1,
\end{align}
since \(\mathcal{R}_0(t) = \mathbb{V}_m \star \mathcal{E}_1^{(m)} \star {\rm \bf r}_1\).
\begin{proposition}\label{residufom}
	Let $\mathcal{A} \in \mathbb{R}^{n \times n \times n_3}$ and ${\mathcal{B}} \in \mathbb{R}^{n \times 1 \times n_3}$. The residual ${\mathcal{R}}_m(t) = {\mathcal{B}} - (\mathcal{A} + t \, \mathcal{I}_{n\times n\times n_3}) \star {\mathcal{X}}_m(t)$ associated with the approximate solution of \eqref{shs} obtained after $m$ steps of the tensor t-tubal Arnoldi method is given by
	\begin{equation}\label{RRm}
		{\mathcal{R}}_m(t) = -{\mathcal{V}}_{m+1} \star ({\tt h}_{m+1,m} \star \mathcal{E}_m) \star \mathcal{Z}_m(t),
	\end{equation}
	where ${\tt h}_{m+1,m}$ and $\mathcal{E}_m$ are defined below \eqref{decompp}.
\end{proposition}

\begin{proof}
	We use equation \eqref{reSy} in this proof and the fact that ${\mathcal{R}}_0(t) = \mathbb{V}_m \star \mathcal{E}_1^{(m)} \star {\rm \bf r}_1$:
	\begin{align*}
		{\mathcal{R}}_m(t) &= {\mathcal{B}} - (\mathcal{A} + t \, \mathcal{I}_{n\times n\times n_3}) \star {\mathcal{X}}_m(t) \\
		&= {\mathcal{B}} - (\mathcal{A} + t \, \mathcal{I}_{n\times n\times n_3}) \star \left( {\mathcal{X}}_0 + \mathbb{V}_m \star \mathcal{Z}_m(t) \right) \\
		&= \mathcal{R}_0(t) - (\mathcal{A} + t \, \mathcal{I}_{n\times n\times n_3}) \star \mathbb{V}_m \star \mathcal{Z}_m(t) \\
		&= \mathcal{R}_0(t) - \mathbb{V}_m \star (\mathcal{H}_m + t \, \mathcal{I}_{m\times m\times n_3}) \star \mathcal{Z}_m(t) - {\mathcal{V}}_{m+1} \star ({\tt h}_{m+1,m} \star \mathcal{E}_m) \star \mathcal{Z}_m(t) \\
		&= \mathcal{R}_0(t) - \mathbb{V}_m \star \mathcal{E}_1^{(m)} \star {\rm \bf r}_1 - {\mathcal{V}}_{m+1} \star ({\tt h}_{m+1,m} \star \mathcal{E}_m) \star \mathcal{Z}_m(t) \\
		&= - {\mathcal{V}}_{m+1} \star ({\tt h}_{m+1,m} \star \mathcal{E}_m) \star \mathcal{Z}_m(t).
	\end{align*}
	This completes the proof.
\end{proof}

\noindent   In particular, if the initial guess $\mathcal{X}_0$ is set to be zero, the approximate solution ${\mathcal{X}}_m(t)$ of the tensor shifted linear system \eqref{shs} is expressed as
\begin{equation}\label{tfomshift}
	{\mathcal{X}}_m(t) = \mathbb{V}_m \star (\mathcal{H}_m + t \mathcal{I}_{m\times m\times n_3})^{-1} \star \mathcal{E}_1^{(m)} \star {\rm \bf r}_1, \quad t > 0.
\end{equation}

\begin{theorem}
	Let $\mathcal{A} \in \mathbb{R}^{n \times n \times n_3}$ and ${\mathcal{B}} \in \mathbb{R}^{n \times 1 \times n_3}$. Then, the approximate expression of $f(\mathcal{A}) \star \mathcal{B}$ for the Stieltjes function after $m$ steps of the tensor t-tubal Arnoldi method is given by
	\begin{equation}\label{FFFm}
		{\mathcal{F}}_m = \mathbb{V}_m \star f(\mathcal{H}_m) \star \mathcal{E}_1^{(m)} \star {\rm \bf r}_1.
	\end{equation}
	Moreover, the $m$th error of this approximation is
	\begin{align}\label{error}
		f(\mathcal{A}) \star {\mathcal{B}} - \mathcal{F}_m = \int_{0}^{\infty} (\mathcal{A} + t \, \mathcal{I}_{n\times n\times n_3})^{-1} \star {\mathcal{R}}_m(t) \, d\mu(t),
	\end{align}
	where ${\mathcal{R}}_m(t)$ is defined in \eqref{RRm}.
\end{theorem}

\begin{proof}
	Starting with the definition of \( \mathcal{F}_m \) in \eqref{FFFm} and the expression for \( \mathcal{X}_m(t) \) in \eqref{tfomshift}, we proceed as follows:
	
	\begin{align*}
		{\mathcal{F}}_m &= \int_{0}^{\infty} {\mathcal{X}}_m(t) \, d\mu(t) \\
		&= \int_{0}^{\infty} \mathbb{V}_m \star (\mathcal{H}_m + t\mathcal{I}_{m\times m\times n_3})^{-1} \star \mathcal{E}_1^{(m)} \star {\rm \bf r}_1 \, d\mu(t) \\
		&= \mathbb{V}_m \star \int_{0}^{\infty} (\mathcal{H}_m + t\mathcal{I}_{m\times m\times n_3})^{-1}\star \mathcal{E}_1^{(m)} \star {\rm \bf r}_1 \, d\mu(t) \\
		&= \mathbb{V}_m \star f(\mathcal{H}_m) \star \mathcal{E}_1^{(m)} \star {\rm \bf r}_1.
	\end{align*}
	
	For the $m$th error \eqref{error}, we have
	\begin{align*}
		f(\mathcal{A}) \star {\mathcal{B}} - {\mathcal{F}}_m &= \int_{0}^{\infty} (\mathcal{A} + t \, \mathcal{I} _{n\times n\times n_3})^{-1} \star {\mathcal{B}} \, d\mu(t) \\
		&\quad - \int_{0}^{\infty} \mathbb{V}_m \star (\mathcal{H}_m + t \, \mathcal{I} _{m\times m\times n_3})^{-1} \star \mathcal{E}_1^{(m)} \star {\rm \bf r}_1 \, d\mu(t) \\
		&= \int_{0}^{\infty} (\mathcal{A} + t \, \mathcal{I} _{n\times n\times n_3})^{-1} \star {\mathcal{B}} - {\mathcal{X}}_m(t) \, d\mu(t) \\
		&= \int_{0}^{\infty} (\mathcal{A} + t \, \mathcal{I} _{n\times n\times n_3})^{-1} \star \left( {\mathcal{B}} - (\mathcal{A} + t \, \mathcal{I} _{n\times n\times n_3}) \star {\mathcal{X}}_m(t) \right) \, d\mu(t) \\
		&= \int_{0}^{\infty} (\mathcal{A} + t \, \mathcal{I} _{n\times n\times n_3})^{-1} \star {\mathcal{R}}_m(t) \, d\mu(t).
	\end{align*}
\end{proof}

\noindent In the following, we show that the approximation is exact for polynomials with tube-fiber coefficients of degree at most $m-1$. Before proceeding, we provide the definition of polynomials with tube-fiber coefficients.
\begin{definition} 
	Let $\left\lbrace {\tt a}_1,\ldots,{\tt a}_m\right\rbrace $ be a set of tube fibers ${\tt a}_i \in \mathbb{R}^{1\times 1\times n_3}$. The polynomials of degree $m$ with tube fibers coefficients are defined as follows:
	\begin{equation*}
		p ({\mathcal{X}})= \sum_{i=0}^{n} {\mathcal{X}}^i \star {\tt a}_i ,  \; \; {\mathcal{X}} \in \mathbb{R}^{n\times 1\times n_3}.
	\end{equation*}
	We denote by $\mathbb{P}_m$ the set of polynomials of degree $m$ with tube-fibers coefficients.
\end{definition} 

\noindent Notice that the polynomials of degree $m$ with tube-fibers coefficients can be considered as the generalization of the well-known vector polynomials with scalar coefficients (the tube fibers play the role of scalars). 

\noindent Given a tube-fiber polynomial $p \in \mathbb{P}_m$, we define the action of a tensor $\mathcal{A}\in \mathbb{R}^{n\times n\times n_3}$ with ${\mathcal{V}}\in \mathbb{R}^{n\times 1\times n_3}$ using the operator $\circ$ as follows:
\begin{align*}
	p  (\mathcal{A})\circ {\mathcal{V}} =
	\sum_{i=1}^m   \left(
	\mathcal{A}^{i-1}\star{\mathcal{V}}\right)\star{\rm \bf a}_{i}.
\end{align*}

\noindent There is an inherent connection between the Krylov subspace $\mathcal{TK} _{m}(\mathcal{A},{\mathcal{V}} )$ defined in equation \eqref{ttgk} and the notion of polynomials with tube fibers coefficients. In fact, we have:
\begin{align*}
	\mathcal{TK}_{m}(\mathcal{A},{\mathcal{V}} ) &= \text{Tubal-Span}\left\lbrace {\mathcal{V}},\mathcal{A} \star {\mathcal{V}}, \ldots, \mathcal{A}^{m-1} \star {\mathcal{V}} \right\rbrace \\
	&= \left\lbrace \mathcal{Z} \in \mathbb{R}^{n\times 1\times n_3}, \mathcal{Z}= \sum_{i=1}^m \left(
	\mathcal{A}^{i-1} \star {\mathcal{V}} \right)\star{\rm \bf a}_{i} \right\rbrace \\
	&= \left\lbrace p_{m-1}  (\mathcal{A}) \circ {\mathcal{V}} : p \in \mathbb{P}_{m-1} \right\rbrace .
\end{align*}

\begin{theorem}
	Let $\mathcal{A} \in \mathbb{R}^{n \times n \times n_3}$ and ${\mathcal{B}} \in  \mathbb{R}^{n \times 1 \times n_3}$. Assume that $m$ steps of the tensor t-tubal Arnoldi method (Algorithm \ref{TTA}) have been applied to the pair $(\mathcal{A}, \mathcal{B})$. Then, for any polynomial with tube-fibers coefficients $p_{m-1} \in \mathbb{P}_{m-1}$, it holds:
	\begin{equation*}
		p_{m-1}(\mathcal{A}) \circ {\mathcal{B}} = \mathbb{V}_m \star (p_{m-1} (\mathcal{H}_m) \circ (\mathcal{E}_1^{(m)} \star {\rm \bf r}_1)) .
	\end{equation*}
\end{theorem}

\begin{proof}
	It suffices to show that
	$$ \mathcal{A}^j \star {\mathcal{B}} = \mathbb{V}_m \star \mathcal{H}_m^j \star \mathcal{E}_1^{(m)} \star {\rm \bf r}_1, \quad j = 0, \ldots, m-1.$$
	\noindent We have $\mathcal{A}^0 \star {\mathcal{B}} = {\mathcal{B}} = {\mathcal{V}}_1 \star {\rm \bf r}_1 = \mathbb{V}_m \star \mathcal{E}_1^{(m)} \star {\rm \bf r}_1$. Let $j=2, 3, \ldots, m-1$, and assume that
	$$ \mathcal{A}^k \star {\mathcal{B}} = \mathbb{V}_m \star \mathcal{H}_m^k \star \mathcal{E}_1^{(m)} \star {\rm \bf r}_1, \quad k = 0, \ldots, j-1.$$
	We will show
	$$ \mathcal{A}^j \star {\mathcal{B}} = \mathbb{V}_m \star \mathcal{H}_m^j \star \mathcal{E}_1^{(m)} \star {\rm \bf r}_1.$$
	By induction, we have:
	$$ \mathcal{A}^j \star {\mathcal{B}} = \mathcal{A} \star \mathcal{A}^{j-1} \star {\mathcal{B}} = \mathcal{A} \star \mathbb{V}_m \star \mathcal{H}_m^{j-1} \star \mathcal{E}_1^{(m)} \star {\rm \bf r}_1.$$
	Using the decomposition (\ref{decompp}), we get:
	\begin{align*}
		\mathcal{A}^{j} \star {\mathcal{B}} &= [\mathbb{V}_{m} \star \mathcal{H}_{m} + \mathcal{V}_{m+1} \star (
		{\tt h}_{m+1,m} \star \mathcal{E}_{m})] \star \mathcal{H}_m^{j-1} \star \mathcal{E}_1^{(m)} \star {\rm \bf r}_1 \\
		&= \mathbb{V}_{m} \star \mathcal{H}_{m} \star \mathcal{H}_m^{j-1} \star \mathcal{E}_1^{(m)} \star {\rm \bf r}_1 + \mathcal{V}_{m+1} \star (
		{\tt h}_{m+1,m} \star \mathcal{E}_{m}) \star \mathcal{H}_m^{j-1} \star \mathcal{E}_1^{(m)} \star {\rm \bf r}_1.
	\end{align*}
	\noindent Since $\mathcal{H}_m$ is an upper Hessenberg tensor (each frontal slice is an upper Hessenberg matrix), we obtain:
	$$ \mathcal{E}_{m} \star \mathcal{H}_m^j \star \mathcal{E}_1^{(m)} = 0, \;\; j = 1, \ldots, m-2.$$
	Then, it follows that:
	\begin{align*}
		\mathcal{A}^{j} \star {\mathcal{B}} = \mathbb{V}_{m} \star \mathcal{H}_{m} \star \mathcal{H}_m^{j-1} \star \mathcal{E}_1^{(m)} \star {\rm \bf r}_1
		= \mathbb{V}_{m} \star \mathcal{H}_m^{j} \star \mathcal{E}_1^{(m)} \star {\rm \bf r}_1.
	\end{align*}
	This concludes the proof.
\end{proof}

\begin{algorithm}[H]
	\caption{Tensor t-Tubal Arnoldi Method (TTA) for Approximating a  t-function} \label{TTGGMRES}
	
	{\bf Input:} $\mathcal{A} \in \mathbb{R}^{n \times n \times n_3}$, ${\mathcal{B}} \in \mathbb{R}^{n \times 1 \times n_3}$, a function $f$ such that $f(\mathcal{A})$ is defined, and an integer $m$.
	
	\begin{enumerate}
		\item Compute the tensor t-tubal Arnoldi decomposition given by \eqref{decompp}:
		\[
		\mathcal{A} \star \mathbb{V}_{m} = \mathbb{V}_{m} \star \mathcal{H}_{m} + {\mathcal{V}}_{m+1} \star \left(
		{\tt h}_{m+1,m} \star \mathcal{E}_{m}\right),
		\]
		using Algorithm \ref{TTA}.
		
		\item Compute:
		\[
		{\mathcal{F}}_m = \mathbb{V}_m \star f(\mathcal{H}_m) \star \mathcal{E}_1^{(m)} \star {\bf r}_1.
		\]
	\end{enumerate}
	
	{\bf Output:} Approximation of $f(\mathcal{A}) \star \mathcal{B}$.
\end{algorithm}

\section{Application to the approximation of the tensor t-exponential }

In this section, we consider the solution to a multidimensional ordinary differential equation (t-ODEs).

Let $\mathcal{U}\colon [0, \infty) \rightarrow \mathbb{R}^{n \times 1\times n_3}$ be an unknown function, such that $\mathcal{U}(t)$ is the solution of the following evolutionary equation:
\begin{align}\label{tODEs}
	\mathcal{U}^{'}(t)=-\mathcal{A}\star \mathcal{U}(t),
\end{align}
with $\mathcal{U}(0):=\mathcal{B}\in\R^{n\times 1\times n_3},$ where the derivative acts element-wise. Then, from \eqref{Expo}, $\mathcal{U}(t)$ can be expressed using the tensor exponential as follows:
\begin{align*}
	\mathcal{U}(t):=\exp(-t\mathcal{A})\star \mathcal{B}.
\end{align*}

Based on the definition of the tensor Krylov subspace $\mathcal{TK}_{m}(\mathcal{A}, \mathcal{V})$ provided in \eqref{krylovexpr2}, the $m$-step approximation of $\mathcal{U}(t)$ obtained through the tensor t-tubal Arnoldi process is given by:
\[
\mathcal{U}_m(t)=\mathbb{V}_m\star \vecc{\mathcal{Y}}_m(t),\quad \vecc{\mathcal{Y}}_m(t)\in \mathbb{R}^{m\times 1\times n_3}.
\]

$\vecc{\te{Y}}_m$ is obtained such that the residual $\vecc{\te{R}}_m(t) = \te{U}'_m(t) + \te{A}\te{U}_m(t)$ with respect to the (t-ODEs) in \eqref{tODEs} is $T$-orthogonal to $\mathcal{TK}_{m}(\mathcal{A}, {\mathcal{V}})$, i.e.,

\[
\mathbb{V}_m\star \mathcal{R}_m(t)=0,\quad \forall t>0.
\]

We have $\mathcal{U}^{'}_m(t)=\mathbb{V}_m\star \mathcal{Y}^{'}_m(t)$, $\mathcal{H}_m=\mathbb{V}_m^T\star \mathcal{A}\star \mathbb{V}_m$ and $\mathbb{V}_m^T\star \mathbb{V}_m=\mathcal{I}_m$. Then it follows that 
\[
\mathcal{Y}_m^{'}(t)+\mathcal{H}_m\star \mathcal{Y}_m(t)=0\quad t>0,
\]
where $\{\mathbb{V}_m,\mathcal{H}_m\}$ is a pair of tensors generated by $m$ steps of TTA algorithm, (Algorithm \ref{TTA}) when applied to $\mathcal{A}$ and $\mathcal{B}$. Furthermore, we have 
\[
\mathcal{U}_m(0)=\mathbb{V}_m\star \mathcal{Y}_m(0)=\mathcal{U}(0)=B=\mathbb{V}_m\star \mathcal{E}^{(m)}_1\star r_1.
\]
$\mathcal{E}^{(m)}_1$ and $r_1$ are defined below \eqref{decompp}. The approximation $\mathcal{Y}_m$ satisfies the initial condition $\mathcal{Y}_m(0)=\mathcal{E}^{(m)}_1\star r_1.$ Hence, $\mathcal{Y}_m(t)$ solves the reduced multidimensional tensor ODE 
\begin{align}\label{redODE}
	\begin{cases}
		& \mathcal{Y}_m^{'}(t)+\mathcal{H}_m\star \mathcal{Y}_m(t)=0,\quad t>0\\
		& \mathcal{Y}_m(0)=\mathcal{E}^{(m)}_1\star r_1,
	\end{cases}
\end{align}
which gives \[
\mathcal{Y}_m(t)=\exp(-t\mathcal{H}_m)\star (\mathcal{E}^{(m)}_1\star r_1),\quad t>0.
\]

Therefore, the approximate solution of $\mathcal{U}_m(t)$ is written as
\begin{align*}
	\mathcal{U}_m(t)=\mathbb{V}_m\star \exp(-t\mathcal{H}_m)\star (\mathcal{E}^{(m)}_1\star r_1),\quad t>0.
\end{align*}

The next result provides a formula for the residual norm of $\mathcal{R}_m(t)$, formulated so as to avoid any explicit matrix?tensor products involving the large tensor $\te{A}$.

\begin{proposition}
	Let $\mathcal{Y}_m(t)$ be the solution of the reduced ($t$-ODEs) in \eqref{redODE} and let $\mathcal{U}_m(t)$ be the approximation of $\mathcal{U}(t)$ obtained after $m$ steps of the TTA algorithm (Algorithm \ref{TTA}). Then, the residual $\mathcal{R}_m(t)$ is given by: 
	\begin{align*}
		\|\mathcal{R}_m(t)\|_F=\|(h_{m+1,m}\star \mathcal{E}_m)\star \mathcal{Y}_m(t)\|_F.
	\end{align*}
\end{proposition}

\begin{proof}
	We have
	\begin{align*}
		\mathcal{R}_m(t)&=\mathcal{U}'_m(t)+\mathcal{A}\star \mathcal{U}_m(t)\\
		&=\mathbb{V}_m\star \mathcal{Y}'_m(t)+\mathcal{A} \star\mathbb{V}_m\star\mathcal{Y}_m(t)\\
		&=\mathbb{V}_m\star( \mathcal{Y}'_m(t)+\mathcal{H}_m\star\mathcal{Y}_m(t))+\mathcal{V}_{m+1}\star(h_{m+1,m}\star\mathcal{E}_m)\star\mathcal{Y}_m(t)\\
		\intertext{and because  $\mathcal{Y}_m(t)$ is the solution of \eqref{redODE}, it follows that}
		\mathcal{R}_m(t)&=\mathcal{V}_{m+1}\star(h_{m+1,m}\star\mathcal{E}_m)\star\mathcal{Y}_m(t)
	\end{align*}
	Therefore,
	\[
	\|\mathcal{R}_m(t)\|_F=\|\mathcal{V}_{m+1}\star(h_{m+1,m}\star\mathcal{E}_m)\star\mathcal{Y}_m(t)\|_F=\|(h_{m+1,m}\star\mathcal{E}_m)\star\mathcal{Y}_m(t)\|_F.
	\]
\end{proof}

Adopting the strategy introduced by Van den Eshof and Hochbruck in \cite{VH} for the exponential function, we consider the following error estimate to analyze the convergence behavior of the $(m+1)$-step tensor t-tubal Arnoldi approximation of $\mathcal{U}(t)$.

\begin{align}\label{Estimate}
	\epsilon_{m}(t):= \dfrac{\|\mathcal{U}(t)-\mathcal{U}_{m}(t)\|_F}{\|\mathcal{U}_{m}(t)\|_F}\approx \dfrac{\delta_{m}(t)}{1-\delta_{m}(t)}:=\gamma_{m}(t),\quad \forall t>0.
\end{align}
where $\delta_{m}:=\|\mathcal{U}_{m+1}(t)-\mathcal{U}_{m}(t)\|_F/\|\mathcal{U}_{m}(t)\|_F.$ The error estimate was also applied in \cite{KS} in the context of the extended Arnoldi method. The quantity $\gamma_{m}(t)$ can be considered as a stopping criterion for detecting convergence.
The computation of $\gamma_{m}(t)$ requires performing one additional step of the algorithm, but it is useful for estimating the error at step $m$.
Define the $\infty$-estimate error and $\infty$-relative error by
\begin{align}\label{infestimate}
	\epsilon_{m}:=\max\limits_{t_1,\ldots,t_M}\epsilon_{m}(t),\quad \text{and}\quad  \gamma_{m}:=\max\limits_{t_1,\ldots,t_M}\gamma_{m}(t). 
\end{align}

\begin{algorithm}
	\caption{Approximation of $\exp{(-t\mathcal{A})}\star \mathcal{B}$ for multiple values of $\{t_1,\ldots,t_M\}$ by the  tensor t-tubal Arnoldi method (exp-TTA)}\label{alg:exponential}
	\textbf{Inputs:}  Tensors $\mathcal{A}$,  $\mathcal{B}$ and $\Sigma=\{t_1,\ldots,t_M\}$.
	\begin{enumerate}
		\item Choose a tolerance $tol>0$, and a maximum number of $itermax$ iterations.
		\item  Set   $\gamma_0=1$ and $\Sigma_c=\emptyset$.
		\item Set $[\mathcal{V}_{1}, {\rm \bf r}_{1}]=  \text{Normalization}(\mathcal{V})$
		\item For $m=1:itermax$
		\begin{enumerate}
			\item $\widetilde{\mathcal{W}}=\mathcal{A}\star \mathcal{V}_{m};$
			\item For $i=1:m$\\
			$h_{i,m}=\mathcal{V}_i^T\star \widetilde{\mathcal{W}};$\quad $\widetilde{\mathcal{W}}=\widetilde{\mathcal{W}}-\mathcal{V}_i\star h_{i,m};$\\
			endfor
			\item $[\mathcal{V}_{m+1}, {\rm \bf h}_{m+1,m}]=  \text{Normalization}(\widetilde{\mathcal{W}})$

			\item Compute $\mathcal{Y}_m(t)=\exp(-t\mathcal{H}_m)\star (\mathcal{E}^{(m)}_1\star r_1)$ and  $\mathcal{U}_{m}(t)=\mathbb{V}_{m}\star\mathcal{Y}_{m}(t)$ for $t\in \Sigma/\Sigma_c$.
			\item Select the new $t\in \Sigma/\Sigma_c$ such that  $\gamma_{m-1}(t)<tol$, where $\gamma_{m-1}(t)=\delta_{m-1}(t)/(1-\delta_{m-1}(t))$ is given by \eqref{Estimate}. Update set $\Sigma_c$ of converged $\mathcal{U}_{m-1}(t)$.
			\item if $\Sigma_c=\Sigma$. \\
			Break.\\ end if.
			
			\item end for.
		\end{enumerate}
	\end{enumerate}
	\textbf{Outputs:} Approximation $\mathcal{U}_{m}(t)$ of $\mathcal{U}(t)=\exp{(-t\mathcal{A})}\star \mathcal{B}$\, for $t=\{t_1,\ldots,t_M\}$  and the $\infty$-estimate error  $\gamma_{m-1}=\max\limits_{t\in \{t_1,\ldots,t_M\}}\gamma_{m-1}(t)$.
\end{algorithm}

Algorithm \ref{alg:exponential} describes the approximation of $\mathcal{U}(t) = \exp{(-t\mathcal{A})} \star V$. It is an iterative algorithm that reuses information for $\{t_1, \ldots, t_M\}$. In this algorithm, we start with $\Sigma_c = \emptyset$, and at each iteration $m$, we identify the values $t \in \Sigma$ such that $\gamma_{m-1}(t) < tol$, where $\gamma_{m-1}(t)$ is given by \eqref{Estimate}. We add the values of $t$ that satisfy this stopping criterion to the set $\Sigma_c$. This indicates that the approximate solutions $\mathcal{U}_{m-1}(t)$ have converged for all values $t \in \Sigma_c$. The iterations terminate when $\Sigma_c = \Sigma$, i.e., when the approximate solutions determined by the tensor t-tubal Arnoldi method have converged for all $t \in \{t_1, \ldots, t_M\}$.

\section{Numerical experiments }
\medskip
\noindent This section presents some numerical results that illustrate the performance of the tensor t-tubal Arnoldi method.     All experiments were carried
out in  MATLAB R2018b on a computer with an  Intel(R)  Core i$7$ processor and $16$GB of RAM. The computations were done with about $15$ significant decimal
digits. The numerical section includes a comparison of the proposed method with two variants of the Arnoldi method, both of which use standard Krylov subspaces. The first approach is based on the approximation of $ f(\bcirc(\mathcal{A})) \, \text{unfold}(\mathcal{B})$ derived directly from the definition of the $t$-function given by \eqref{tensor function}, achieved by applying the standard Arnoldi method to the pair $\{\bcirc(\mathcal{A}), \text{unfold}(\mathcal{B})\}$. This method is denoted as (SA($\bcirc$)). The second approach is based on the fact that $\bcirc(\mathcal{A})$ can be block diagonalized by using the Fourier transform, i.e., 
\[
\bcirc(\mathcal{A})=(F^H_{n_3}\otimes I_n)\boldsymbol{D}(F_{n_3}\otimes I_n).
\]

\noindent By applying $f$ to both sides, we obtain

\[
f(\bcirc(\mathcal{A}))=(F^H_{n_3}\otimes I_n)f(\boldsymbol{D})(F_{n_3}\otimes I_n),
\]
where
\[
\boldsymbol{D}:=\left[
\begin{array}{cccc}
	{A}_1& &&\\
	& {A}_2&&\\
	&&\ddots&\\
	&&&{A}_{n_3}\\
\end{array}
\right]
\]
Thus, $f(\bcirc(\mathcal{A}))\unfold(\mathcal{B})$ can be rewritten as

\[
f(\bcirc(\mathcal{A}))\unfold(\mathcal{B})=(F^H_{n_3}\otimes I_n)f(\boldsymbol{D})(F_{n_3}\otimes I_n)\unfold(\mathcal{B}).
\]

This approach involves the approximation of the following complex-valued matrix function \(f(\boldsymbol{D})(F_{n_3}\otimes I_n)\unfold(\mathcal{B}) \) using the standard Arnoldi method applied to the pair  $(\boldsymbol{D},(F_{n_3}\otimes I_n)\unfold(\mathcal{B}))$. 
We refer to this approach as the standard complex Arnoldi method (SA(fft)).

\subsection{Example for approximations of $f(\mathcal{A})\star \mathcal{B}$}
\noindent In the following example, we compute approximations of \( f(\mathcal{A}) \star \mathcal{B} \) using the TTA method as outlined in Algorithm \ref{TTGGMRES}, along with the SA(fft) and SA(\(\bcirc\)) methods, where \( \mathcal{A} \in \mathbb{R}^{n \times n\times n_3} \) with \( n =10000\).
The tensor $\mathcal{B} \in \R^{n \times 1\times n_3}$ is generated randomly, with entries uniformly distributed in the interval $[0, 1]$. We report the magnitude of the relative error:
\[
Err_m= \frac{|| f(\mathcal{A})\star \mathcal{B} - \mathcal{F}_m||_{F}}{|| f(\mathcal{A})\star \mathcal{B}||_{F}}.   
\]

Here, $\mathcal{F}_m$ is defined in \eqref{FFm}.

\textbf{Example $1$.} 
The tensor $\mathcal{A}$ has three frontal slices (i.e., $n_3 = 3$), derived from the following matrices:
\begin{align*}
	\mathcal{A}^{(1)} &: \textit{a tridiagonal matrix defined as } n^2 \, \textsc{tridiag}(1, 2, -1), \\
	\mathcal{A}^{(2)} &: \textit{a nonsymmetric Toeplitz matrix with first row and column  }\\ &[1,1/2,\ldots,1/1000]\textit{ and } [1,\ldots,1], \textit{ respectively. } \\
	\mathcal{A}^{(3)} &: \textit{a nonsymmetric matrix whose $n$ eigenvalues are log-uniformly}\\
	&\textit{ distributed in the interval } [10^{-1}, 10^4] \textit{ with random
		eigenvectors}.
\end{align*}

Table \ref{tab:table12} displays the magnitude errors $(Err(m))$ and the required CPU time (Time) in seconds for various choices of the function $f$, with the number of iterations set to $m=10$ for the three methods. The results show that all three methods yield approximations of about the same quality. However, the TTA method requires less CPU time than the SA(fft) and SA(\(\bcirc\)) methods.

\begin{table}[htbp]
	\centering
	
		\caption{Example $1$: Results for several functions $f$, $m=10,$ $n=10000$}\label{tab:table12}
		\begin{tabular}{c|c|c|c|c|c|c}
			\multirow{2}{2cm}{$f(x)$} & \multicolumn{2}{c|}{TTA} &
			\multicolumn{2}{c|}{SA(fft)} & \multicolumn{2}{c}{SA(bcric)}\\
	\cmidrule(lr){2-7}	
			& Err$(m)$ & Time & Err$(m)$ & Time &Err$(m)$ & Time\\
			\hline
			\hline
			$\sqrt{x}$ & $1.63\cdot 10^{-6}$ & $14.26$ & 
			$1.77\cdot 10^{-6}$ & $92.13$ & 
			$1.77\cdot 10^{-6}$ & $28.92$ \\
			$\ln{x}$ & $8.18\cdot 10^{-7}$ & $14.66$ &
			$8.89\cdot 10^{-7}$ & $94.76$  &
			$8.89\cdot 10^{-7}$ & $28.54$\\
			$x^{1/4}$ & $1.94\cdot 10^{-6}$ & $15.46$ &
			$2.11\cdot 10^{-6}$ & $97.14$ &
			$2.11\cdot 10^{-6}$ & $29.07$ \\
			$x^{-1/2}$ & $2.42\cdot 10^{-5}$ & $14.19$ & 
			$2.46\cdot 10^{-5}$ & $93.58$ & 
			$2.64\cdot 10^{-5}$ & $28.05$ \\
			$x^{-1/4}$ & $7.15\cdot 10^{-6}$ & $14.40$ & 
			$7.78\cdot 10^{-6}$ & $91.47$ & 
			$7.78\cdot 10^{-6}$ & $27.26$ 
		\end{tabular}
\end{table}

\subsection{Application to the solution of time-dependent partial differential equations}

In this subsection, we illustrate the performance of the tensor t-tubal  Arnoldi method when approximating $\mathcal{U}(t)=\exp{(-t\mathcal{A})}\star \mathcal{B}$ for multiple values of $t\in\{t_1,\ldots,t_M\}$ using the ($\exp$-TTA) method as implemented by Algorithm \ref{alg:exponential} with respect to $\exp$-SA(fft) and $\exp$-SA($\bcirc$) methods.  The $\exp$-SA($\bcirc$) method is applied to the pair $\{\bcirc(\mathcal{A}),{\rm unfold}(\mathcal{B})\}$. The $\exp$-SA(fft) method is applied to the  pair $(\boldsymbol{D},(F_{n_3}\otimes I_n)\unfold(\mathcal{B}))$.  The values of the time step $t$ are taken to be $100$ values uniformly distributed in the interval $[10^{-2},1]$.    The different tables in this subsection report the dimension (dim) of larger space as soon as $\gamma_{dim}<tol$, the required CPU time (Time) in seconds, and the $\infty$-estimate error for both  $\exp$-TTA, $\exp$-SA(fft) and $\exp$-SA($\bcirc)$ methods. The $\infty$-estimate error for the $\exp$-TTA method is given in \eqref{infestimate}. The $\infty$-estimate errors for $\exp$-SA(fft) and $\exp$-SA(bcirc) are computed using the same configuration.

\noindent {\bf Example $2$} The purpose of this example is to illustrate the quality of the estimate error $\gamma_{m}(t)$ compared to the relative error $\epsilon_{m}(t)$ given by \eqref{infestimate} when applying the $\exp$-TTA method to approximate $\mathcal{U}(t)=\exp(-t\mathcal{A})\star \mathcal{B}$. The tensor  $\mathcal{A}$ in this example has three frontal slices (i.e., $n_3 = 3$), each obtained from the centered finite difference discretization (CFDD) in two dimensions with homogeneous Dirichlet boundary conditions for the operators:

\begin{align*}
	\mathcal{L}_{\mathcal{A}^{(1)}}(u) &= -\Delta u + \exp(xy) u_x + \sin(xy) u_y + (y^2-x^2)u,\\
	\mathcal{L}_{\mathcal{A}^{(2)}}(u) &= -\Delta u + \exp(-x*y) u_x + \cos(xy) u_y + (x^2-y^2)u,\\
	\mathcal{L}_{\mathcal{A}^{(3)}}(u) &= -\Delta u.
\end{align*}

The spatial domain is the square $[0,1] \times [0,1]$ with $n = n_0^2$, where $n_0$ is the number of internal grid points in each direction. The starting tensor $\mathcal{B}$ has three frontal slices, each of which is the discretization on $[0,1] \times [0,1]$ of the following functions: 
$(x,y) \mapsto xy(1-x)(1-y)$, $(x,y) \mapsto \sin(\pi x)\sin(\pi y)$, and $(x,y) \mapsto xy(1-x)(1-y)\cos(\pi y)$. 

The plots in Figure \ref{figure:DeriveSeconde} display the estimate error (red) and the relative error (blue) for a dimension of $n = 900$ and different values of $t=\{0.25,0.5,0.75\}$   versus the number of iterations. As illustrated in these plots, the estimated error $\gamma_{m}(t)$ coincides at all iterations with the relative error $\epsilon_{m}(t)$ defined in \eqref{infestimate}.

\begin{figure}
	\includegraphics[width=\textwidth,height=8cm]{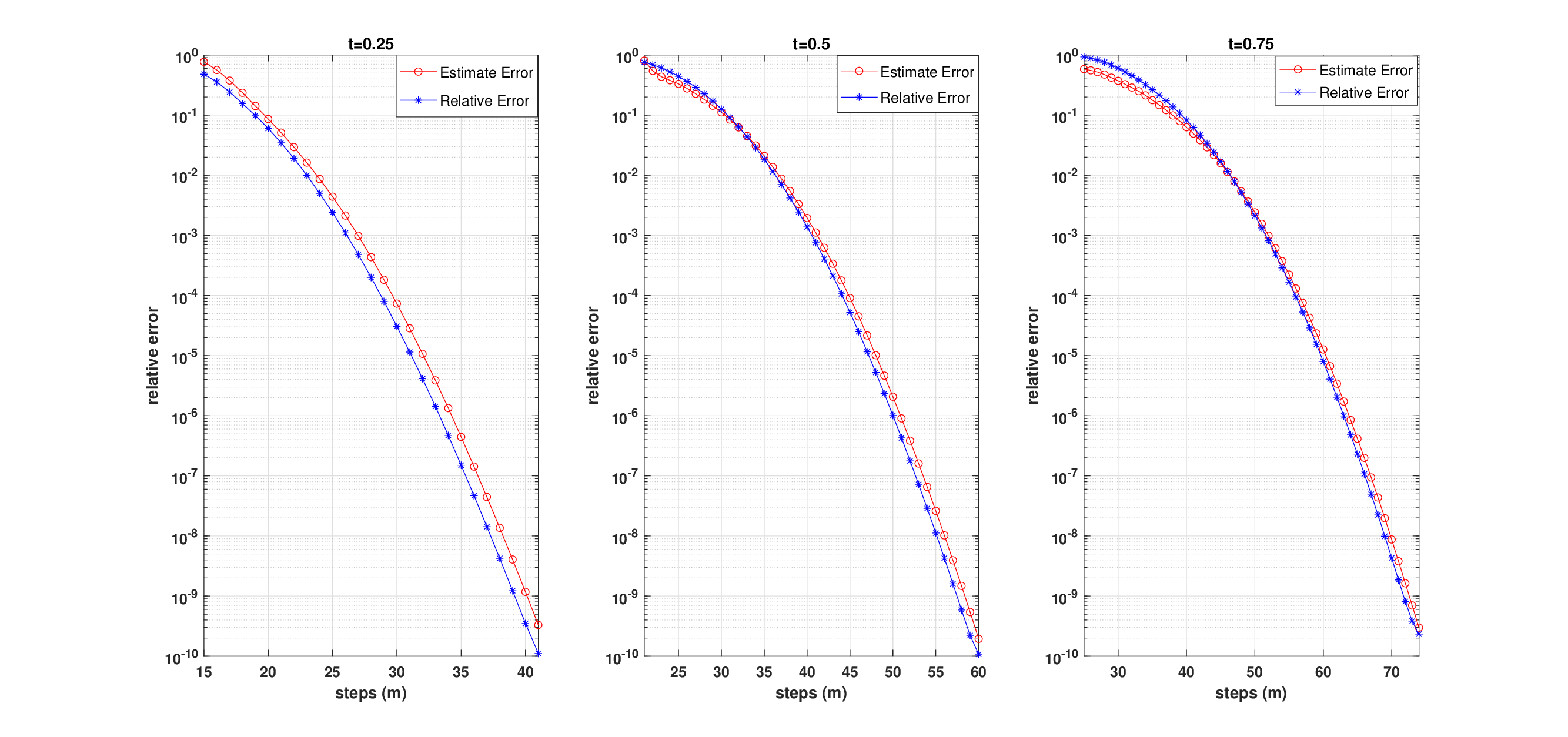}
	\caption{Example $2$: Comparison of the estimate error $\gamma_m(t)$ and the relative error $\epsilon_m(t)$ for different values of $t$}
	\label{figure:DeriveSeconde}
\end{figure}

\noindent {\bf Example $3$}   In this example, we consider the approximation of the solution of the following partial differential equations on the unit square $[0,1] \times [0,1]$ with homogeneous Dirichlet boundary conditions.

\begin{equation*}
	\begin{array}{rll}
		\dfrac{\partial u}{\partial t} - \mathcal{L}(u) &= 0 & \text{on } (0,1)^2 \times (0,1), \\
		u(x,y,t) &= 0 & \text{on } \partial (0,1)^2 \text{ for all } t \in [0,1], \\
		u(x,y,0) &= v & \text{for all } (x,y) \in [0,1]^2.
	\end{array}
\end{equation*}

The number of discretization points in each direction is $n_0$, and the dimension of the matrices is $n = n_0^2$. In this experiment, the matrices $A(\mathcal{L})\in \mathbb{R}^{n \times n}$, with $n \in \{900, 1600, 2500\}$, are coming from the centered finite difference discretization (CFDD) of the operator $\mathcal{L}$, as given by \eqref{EO}.

\begingroup\small
\begin{equation}\label{EO}
	\begin{array}{rl}
		\mathcal{L}(u) &= -\Delta u + \beta\left(\frac{1}{2}(v_1u_x + v_2u_y) + \frac{1}{2}((v_1u)_x + (v_2u)_y)\right).
	\end{array}
\end{equation}
\endgroup

Here, \( \beta \in \mathbb{R} \),  \( v_1 \) and \( v_2 \) are complex-valued functions defined as follows:
\[
v_1(x, y) = x + iy, \quad v_2(x, y) = x - iy, \quad \text{where} \quad i^2 = -1.
\]

The initial vector $v$ is the pointwise discretization of the function $(x,y) \mapsto \sin(\pi x)\sin(\pi y)$ on the internal grid points. 

Next, we define tensors $\mathcal{A}$ and $\mathcal{B}$ with three frontal slices $(n_3=3)$ such that:

\begin{align*}
	\bcirc(\mathcal{A}) &= (F^H_{n_3} \otimes I_n)
	\begin{bmatrix}
		A_1 & & \\
		& A_2 &  \\
		& & A_3  
	\end{bmatrix}
	(F_{n_3} \otimes I_n),
\end{align*}

\begin{align*}
	\unfold(\mathcal{B}) &= (F^H_{n_3} \otimes I_n)
	\begin{bmatrix}
		v  \\
		v  \\
		v  
	\end{bmatrix}.
\end{align*}

where $A_1$ and $A_2$ are coming from the CFDD of the operator $\mathcal{L}$ when $\beta=0$ and $\beta=1$; respectively. $A_3$ is obtained as the conjugate of $A_2$ i.e., $A_3=Conj(A_2)$. In this example,  We let $tol=2\cdot 10^{-5}$ and $itermax=200.$ 

Table~\ref{tab:table2} reports the performance of the $\exp$-TTA, $\exp$-SA(fft), and $\exp$-SA(bcirc) methods for approximating $\exp(-t\mathcal{A}) \star \mathcal{B}$ with increasing values of $n$. All methods achieve comparable $\infty$-estimate errors, on the order of $10^{-5}$. However, the $\exp$-TTA method is faster than the $\exp$-SA(fft) and $\exp$-SA($\bcirc$) methods and it achieves convergence  with significantly lower dimensional projections (fewer iterations).

\begin{table}
	\caption{Example $3$: Approximation of $\exp{(-t\mathcal{A})}\star \mathcal{B}$ for multiple values of $n$ associated to the operator $\mathcal{L}(u)$, $n_3=3.$}
	\label{tab:table2}
	\begin{tabular}{l|c|l|c|c|c}
		Operators & Size$(n)$ & Methods & Time$(s)$ & Dim. & $\infty$- Estimate Error \\
		\hline
		$\mathcal{L}(u)$ & $n=900$ &  $\exp$-TTA & $6.97$ & $46$ & $1.90\cdot 10^{-5}$  \\
		& & $\exp$-SA(fft) & $25.57$ & $106$  & $1.83\cdot 10^{-5}$  \\
		& & $\exp$-SA(bcirc) & $12.97$ & $106$  & $1.99\cdot 10^{-5}$\\
		\hline
		\hline
		$\mathcal{L}(u)$ & $n=1600$ &  $\exp$-TTA & $15.13$ & $60$ & $1.99\cdot 10^{-5}$  \\
		& & $\exp$-SA(fft) & $63.15$ & $138$  & $1.99\cdot 10^{-5}$  \\
		& & $\exp$-SA(bcirc) & $25.70$ & $138$  & $1.99\cdot 10^{-5}$\\
		\hline
		\hline
		$\mathcal{L}(u)$ & $n=2500$ &  $\exp$-TTA & $34.38$ & $74$ & $1.99\cdot 10^{-5}$  \\
		& & $\exp$-SA(fft) & $118.58$ & $169$  & $1.99\cdot 10^{-5}$  \\
		& & $\exp$-SA(bcirc) & $48.22$ & $169$  & $1.99\cdot 10^{-5}$  \\
		
	\end{tabular}
\end{table}

\noindent In Examples 4 and 5, we use real-world network data from the Suite Sparse Matrix Collection \cite{DH}.
Each tensor contains three frontal slices representing adjacency matrices of graphs.
When the matrices $A_1$, $A_2$, and $A_3$ have different dimensions, we set $n = \max{n_1, n_2, n_3}$, where $n_i$ is the dimension of $A_i$, and apply zero-padding to ensure uniform tensor dimensions. The initial tensor $\mathcal{B}$ is chosen as $\mathcal{B}^{(1)} = [1, \ldots, 1]^T \in \mathbb{R}^n$, with the remaining frontal slices are set to zero. Here, $\mathcal{B}^{(1)}$ denotes the first frontal slice of the tensor $\mathcal{B}$.

\noindent \textbf{Example $4$.} In this example, we present the quality of the estimate error $\gamma_{m}(t)$ with respect to the relative error $\epsilon_{m}(t)$ given by \eqref{infestimate} when applying the $\exp$-TTA method to approximate $\mathcal{U}(t) = \exp(-t\mathcal{A}^T) \star \mathcal{B}$. The plots in Figure \ref{figure:directed} display the estimate error (red) and the relative error (blue) for three graphs (Minnesota $n=2642$, Yeast $n=2361$, Delaunay $n=2048$) on the left, and three graphs (CSphd $n=1882$, GD$06\_$Java $n=1538$, Roget $n=1022$) on the right, versus the number of iterations for $t=0.5$.

\begin{figure}
	\includegraphics[width=\textwidth,height=9cm]{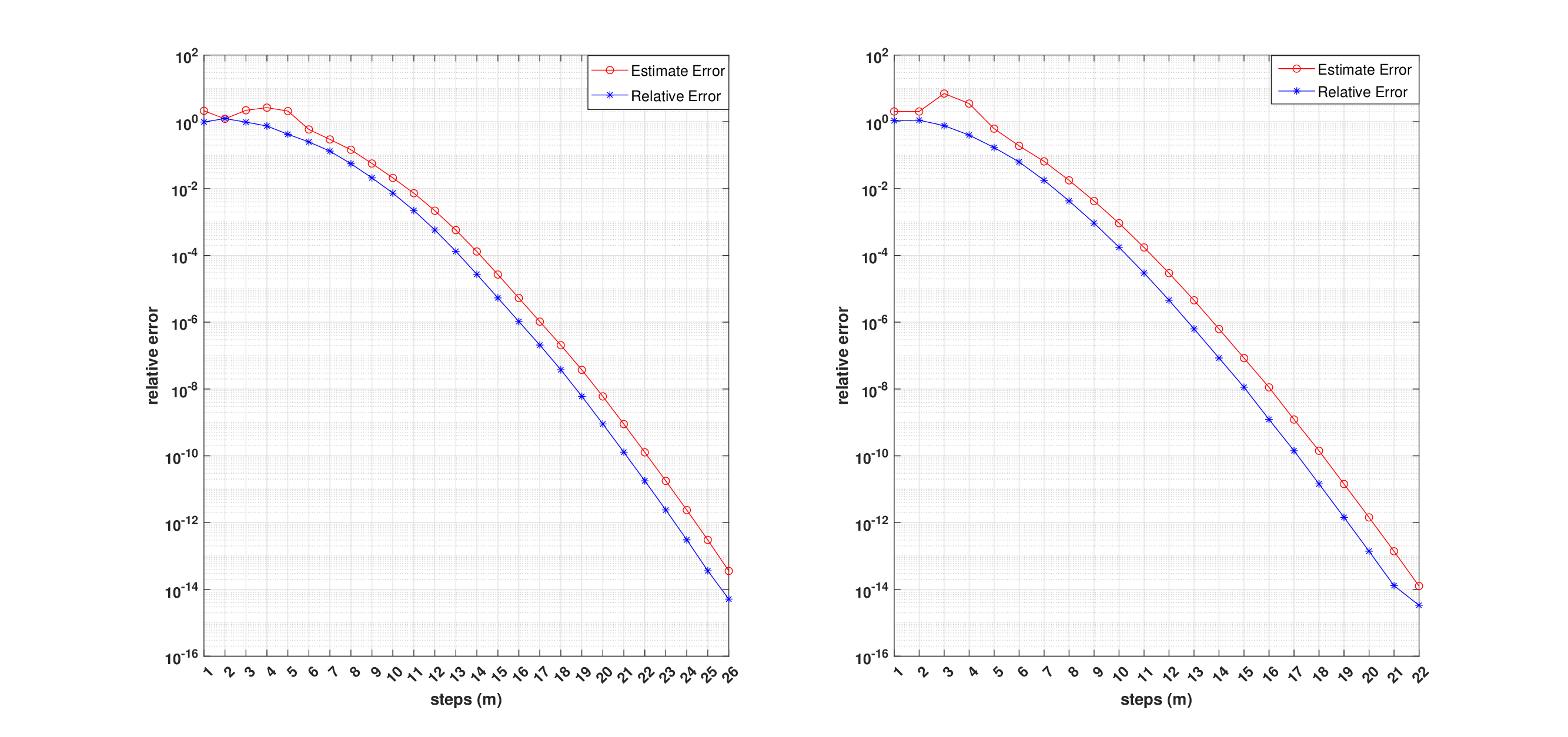}
	\caption{Example $4$: Comparison of the estimate error $\gamma_m(t)$ and the relative error $\epsilon_m(t)$ for $t=0.5$.}
	\label{figure:directed}
\end{figure}

\noindent Figures \ref{figure:directed}  show that the estimate error $\gamma_{m}(t)$ provides higher precision for the behavior of the relative error $\epsilon_{m}(t)$ given by \eqref{infestimate} for $t=0.5$.

\noindent \textbf{Example $5$.} 
In this example, we compare the $\exp$-TTA,  $\exp$-SA(fft) and $\exp$-SA($\bcirc)$ methods for approximating $\exp(-t\mathcal{A}) \star \mathcal{B}$ on several graphs (Oregon-1 with $n=11492$, fe\_4elt2 with $n=11143$, p2p-Gnutella04 with $n=10879$, California with $n=9664$, wb-cs-standard with $n=9914$). We will use two
$n\times n\times 3$ tensors $\mathcal{A}_1$ and $\mathcal{A}_2$ defined as follows

\begin{equation*}
	\begin{array}{lllllll}
		\A_1&=&\fold\Bigg(\begin{bmatrix}\textbf{p2p-Gnutella04}\\\textbf{California}\\
			\textbf{wb-cs-stanford}\end{bmatrix}\Bigg),& \A_2&=&\fold\Bigg(\begin{bmatrix}\textbf{Oregon-1}\\
			\textbf{fe$\_$4elt2}\\\textbf{p2p-Gnutella04}\end{bmatrix}\Bigg).
	\end{array}
\end{equation*}

The approximation errors and timings are presented in Table \ref{tab:table3}. This table shows the results for approximating $\mathcal{U}(t) = \exp{(-t\mathcal{A}^T)} \star \mathcal{B}$ using the $\exp$-TTA method,  We let $tol=2\cdot 10^{-7}$ and $itermax=200.$

\begin{table}
	\caption{Example $5$: Approximation of $\exp{(-t\mathcal{A})}\star \mathcal{B}$ for several graphs, $n_3=3.$}
	\label{tab:table3}
	\begin{tabular}{l|c|l|c|c|c}
		Tensors & Size$(n)$ & Methods & Time$(s)$ & Dim. & $\infty$- Estimate Error \\
		\hline
		$\mathcal{A}_1$ & $n=10879$ &  $\exp$-TTA & $34.12$ & $17$ & $1.99\cdot 10^{-7}$  \\
		& & $\exp$-SA(fft) & $384$ & $28$  & $1.99\cdot 10^{-7}$  \\
		& & $\exp$-SA(bcirc) & $48.96$ & $28$  & $1.99\cdot 10^{-7}$  \\
		$\mathcal{A}_2$ & $n=11492$ & $\exp$-TTA & $41.82$ & $18$ & $1.97\cdot 10^{-7}$  \\
		& & $\exp$-SA(fft) & $410$  & $27$ & $1.98\cdot 10^{-7}$  \\
		& & $\exp$-SA(bcirc) & $55.01$  & $27$ & $1.96\cdot 10^{-7}$  \\
	\end{tabular}
\end{table}

Table~\ref{tab:table3} reports the results for Example~5 on two real-world tensors. All methods yield comparable $\infty$-estimate errors of order $10^{-7}$. The $\exp$-TTA method achieves the lowest computational time and uses smaller projection dimensions compared to $\exp$-SA(fft) and $\exp$-SA(bcirc).

\section*{Conclusion}  
This paper presents the tensor t-tubal Arnoldi method, based on the tensor t-product, for the efficient approximation of matrix functions of the form \eqref{eq1}. Applications include the solution of multidimensional ordinary differential equations of the form 
through estimates of parameter-dependent exponential tensor functions of the form \eqref{expA}. Numerical experiments illustrate the effectiveness of the proposed method in solving well-known time-integration ODEs. The results demonstrate that the tensor Arnoldi method requires fewer iterations and less CPU time than standard Arnoldi matrix methods, to deliver approximations of the same accuracy.


%
\section*{Conflict of interest}

The authors declare no competing interests.

\section*{Data availability}
Data sharing is not applicable to this article as no datasets
were generated or analyzed during the current study.

\section*{Funding}
This research received no specific grant from any funding agency.

\end{document}